\documentclass[a4paper,12pt]{article}
\usepackage{latexsym,amsmath,amsthm,amssymb}
\usepackage{a4wide}
\usepackage{hyperref}
\usepackage{marginnote}
\usepackage{color}
\usepackage{cite}
\hypersetup{
pdftitle={Convex measures}   
pdfauthor={Van Hoang Nguyen},
colorlinks = true,
linkcolor = magenta,
citecolor = blue,
}

\theoremstyle{plain}
\newtheorem{theorem}{Theorem}[section]

\newtheorem{corollary}[theorem]{Corollary}

\theoremstyle{definition}

\theoremstyle{remark}

\renewcommand{\thefootnote}{\arabic{footnote}}
\newcommand{\C}[1]{\ensuremath{{\mathcal C}^{#1}}} 
\newcommand{\ent}{\mathbf{Ent}} 
\newcommand{\var}{\mathbf{Var}} 

\newcommand{\I}{\mathbf I}

\def\R{\mathbb R}
\def\C{\mathbb C}

\def\al{\alpha}
\def\be{\beta}
\def\ga{\gamma}
\def\de{\delta}
\def\De{\Delta} 

\def\lam{\lambda}
\def\vphi{\varphi}
\def\ep{\epsilon}
\def\na{\nabla}
\def\pa{\partial}
\def\la{\langle} 
\def\ra{\rangle} 
\def\lt{\left}
\def\rt{\right}

\def\mvb{\mu_{\varphi,\beta}}

\numberwithin{equation}{section}

\title{$\Phi-$entropy inequalities and asymmetric covariance estimates for convex measures}
\author{Van Hoang Nguyen\footnote{Institute of Research and Development, Duy Tan University, Da Nang, Vietnam.}
}

\begin{document}
\maketitle


\renewcommand{\thefootnote}{}

\footnote{Email: \href{mailto: Van Hoang Nguyen <vanhoang0610@yahoo.com>}{vanhoang0610@yahoo.com}.}

\footnote{2010 \emph{Mathematics Subject Classification\text}: 26D10.}

\footnote{\emph{Key words and phrases\text}: $\Phi-$entropy inequalities, Poincar\'e type inequalities, Beckner type inequalities, semi-group, $L^2-$method of H\"ormander, Brascamp--Lieb type inequalities, asymmetric covariance estimates, convex measures.}

\renewcommand{\thefootnote}{\arabic{footnote}}
\setcounter{footnote}{0}

\begin{abstract}
In this paper, we use the semi-group method and an adaptation of the $L^2-$method of H\"ormander to establish some $\Phi-$entropy inequalities and asymmetric covariance estimates for the strictly convex measures in $\R^n$. These inequalities extends the ones for the strictly log-concave measures to more general setting of convex measures. The $\Phi-$entropy inequalities are turned out to be sharp in the special case of Cauchy measures. Finally, we show that the similar inequalities for log-concave measures can be obtained from our results in the limiting case.
\end{abstract}

\section{Introduction}
Let $\vphi: \R^n \to (0,\infty)$ be a strictly convex, $C^2$ smooth function such that $\vphi^{-\beta}$ is integrable for some $\beta >0$. By strictly convex, we mean that the Hessian matrix, $D^2 \varphi(x) =(\partial^2_{ij} \vphi(x))_{i,j=1}^n$, of $\vphi$ is everywhere positive in the matrix sense. Let $d\mu_{\vphi,\beta}$ denote the probability measure
\[
d\mu_{\vphi,\beta} = \frac{\varphi(x)^{-\beta}}{Z_{\vphi,\beta}} dx,
\]
where $Z_{\vphi,\beta}$ is the normalization constant which turns $\mu_{\vphi,\beta}$ into a probability. The main aims of this paper is to establish several functional inequalities for the probability measure $\mu_{\vphi,\beta}$ such as $\Phi-$ entropy inequalities and asymmetric covariance estimates. These inequalities extend the $\Phi-$entropy inequalities in \cite{BolleyGentil} and the asymmetric covariance estimates in \cite{CCL} for the log-concave measure to the context of convex measures. 

Let $\Phi :I \to \R$ be a convex function on an interval $I \subset \R$ and $f: \R^n \to I$ be a measurable function such that $f$ and $\Phi(f)$ is integrable with respect to the probability measure $\mu_{\vphi,\beta}$, we define
\[
\ent_{\mu_{\vphi,\beta}}^\Phi(f) = \int_{\R^n} \Phi(f) d\mu_{\vphi,\beta} -\Phi\lt(\int_{\R^n} f d\mu_{\vphi,\beta}\rt)
\]
as the $\Phi-$entropy of $f$ under the probability measure $\mvb$. For examples, if $\Phi(x) = x^2$ then we let $\var_{\mvb}(f) = \ent_{\mvb}^\Phi(f)$ be the variance of $f$ with respect to $\mvb$, and if $\Phi(x) = x \ln x$ on $(0,\infty)$ then we let $\ent_{\mvb}(f) = \ent_{\mvb}^\Phi(f)$ be the Boltzmann entropy of a positive function $f$ with respect to $\mvb$. Notice that $\ent_{\mvb}^\Phi(f)$ is always nonnegative quantity by Jensen's inequality. We are interested in to finding the upper bound for $\ent_{\mvb}^\Phi(f)$ under some suitable conditions on $\vphi$, $\Phi$ and $\beta$. The first main result of this paper is the following theorem.

\begin{theorem}\label{Phientropy}
Let $\beta > n+1$ and $\Phi: I \to \R$ be a convex function such that
\begin{equation}\label{eq:Phicondition}
\Phi^{(4)}(t) \Phi''(t) \geq \frac1 8 \frac{(4\beta -5)^2 + n-1}{(\beta-1) (\beta -n -1)} (\Phi^{(3)}(t))^2,
\end{equation}
for any $t \in I$. Assume, in addition, that $\vphi$ is uniformly convex in $\R^n$, i.e., $D^2\vphi(x) \geq c\, \I_n$ in the matrix sence for some $c >0$. Then for any smooth function $f$ with value in $I$, we have
\begin{equation}\label{eq:Phientropyineq}
\ent_{\mvb}^\Phi(f) \leq \frac1{2c(\beta -1)} \int_{\R^n} \Phi''(f) |\na f|^2 \vphi d\mvb.
\end{equation}
\end{theorem}
Let us give some comments on Theorem \ref{Phientropy}. The $\Phi-$entropy inequalities have been proved in \cite{BolleyGentil} for such function $\Phi$ under the curvature-dimension condition $CD(\rho,\infty)$ (see also \cite{Chafai}). Let $L$ be a differential operator of order $2$ given by
\[
Lf(x) = \sum_{i,j=1}^n D_{ij}(x) \frac{\pa^2 f}{\pa x_i \pa x_j}(x) - \sum_{i=1}^n a_i(x) \frac{\pa f}{\pa x_i}(x)
\]
where $D(x) = (D_{ij}(x))_{1\leq i,j \leq n}$ is a nonnegative symmetric $n\times n$ matrix in the matrix sense with smooth entires and $a(x) = (a_i(x))_{1\leq i\leq n}$ has smooth elements. Such an operator generates a semigroup $P_t$ acting on the smooth functions on $\R^n$ such that $L = \lt(\frac{\pa}{\pa t}\rt)_{t=0} P_t$. The \emph{carr\'e du champ} operator (see \cite{BE}) associated to $L$ (or semigroup $P_t$) is defined by
\[
\Gamma(f,g) = \frac{1}{2}\lt(L(fg) - f Lg -g Lf\rt).
\]
For simplicity, we write $\Gamma(f) = \Gamma(f,f)$. The $\Gamma_2$ operator is defined by
\[
\Gamma_2(f) = \frac12\lt(L\Gamma(f)  - 2\Gamma(f,Lf)\rt)
\]
We say that the operator $L$ (or semigroup $P_t$) satisfies the curvature-dimension condition $CD(\rho,\infty)$ for some $\rho \in \R$ if
\[
\Gamma_2(f) \geq \rho \Gamma(f),
\]
for all function $f$. This condition is a special case of the curvature--dimension condition $CD(\rho,m)$ with $\rho \in \R$ and $m\geq 1$ introduced by Bakry and \'Emery \cite{BE}. Let $d\mu = e^{-\psi} dx$ be a probability measure in $\R^n$ with $\psi$ being a convex function such that $D^2\psi(x) \geq \rho \I_n$ for any $x \in \R^n$ for some $\rho >0$, then the operator $L$ defined by
\[
Lf(x) = \Delta f(x) - \la \na \psi(x), \na f(x)\ra,
\]
where $\la \cdot, \cdot \ra$ denotes the scalar product in $\R^n$, satisfies the $CD(\rho,\infty)$ condition. Indeed, it is easy to see that $\Gamma(f,g) = \la \na f, \na g\ra$ and by Bochner--Lichnerowicz formula
\[
\Gamma_2(f) = \|D^2 f\|_{HS}^2 + \la D^2 \vphi(x) \na f(x), \na f(x)\ra,
\]
where $\|\cdot\|_{HS}$ denotes Hilbert-Schmidt norm on the space of symmetric matrices. It was proved by Bolley and Gentil \cite{BolleyGentil} for such measures that the following $\Phi-$entropy inequality with $\Phi$ satisfying $\Phi^{(4)} \Phi'' \geq 2 (\Phi^{(3)})^2$ holds
\begin{equation}\label{eq:PhientropyBG}
\ent_\mu^\Phi(f) \leq \frac1{2\rho} \int_{\R^n} \Phi''(f) |\na f|^2 d\mu.
\end{equation}
It is interesting that the $\Phi-$entropy inequality \eqref{eq:PhientropyBG} can be derived from Theorem \ref{Phientropy} by an approximation process. This will be shown at the end of Sect. \ref{provePhientropy} below. 

Taking the function $\Phi= \Phi_p:= t^{\frac 2p}$ on $(0,\infty)$. The function $\Phi_p$ satisfies the condition \eqref{eq:Phicondition} if 
\begin{equation}\label{eq:conditionp}
1 \leq p \leq p_\beta:= 1+ \frac{4(\be -1)(\beta -n-1)}{4(\beta -1)^2 + 4(3n-2)(\beta -1) + n} < 2.
\end{equation} 
Thus, we obtain the following Beckner-type inequalities for the measures $\mvb$ from Theorem \ref{Phientropy}.

\begin{corollary}\label{Becknertypeineq}
Let $\beta > n+1$ and $D^2\vphi \geq c \I_n$ for some $c >0$. Then for any $p \in [1,p_\beta]$ one has
\begin{equation}\label{eq:Beckner}
\int_{\R^n} f^2 d\mvb - \lt(\int_{\R^n} f^p d\mvb\rt)^{\frac2p} \leq \frac{2-p}{c (\beta -1)}\int_{\R^n} |\na f|^2 \varphi d\mvb,
\end{equation}
for any positive, smooth function $f$.
\end{corollary}
If $\vphi(x) = 1 + |x|^2$, then the probability $d\mu_\beta = \frac1{Z_\beta} (1+ |x|^2)^{-\beta}$, $\beta > \frac n2$ is the generalized Cauchy measures. Notice that $D^2\vphi(x) = 2\I_n$. From Corollary \eqref{Becknertypeineq}, we obtain the following Beckner type inequalities for the Cauchy measures $\mu_\beta$: let $\beta > n+1$ and $p\in [1,p_\beta]$ then it holds
\begin{equation}\label{eq:BecknerCauchy}
\frac1{2-p}\lt(\int_{\R^n} f^2 d\mu_\beta - \lt(\int_{\R^n} f^p d\mu_\beta\rt)^{\frac2p}\rt) \leq \frac{1}{2 (\beta -1)}\int_{\R^n} |\na f|^2 (1+ |x|^2) d\mu_\beta
\end{equation}
for any positive, smooth function $f$. When writing this paper, I learned from the work of Bakry, Gentil and Scheffer \cite{BGS} that the inequality \eqref{eq:BecknerCauchy} can be proved by a different method based on the harmonic extensions on the upper-half plane and probabilistic representation and curvature-dimension inequalities with some negative dimensions. This method was initially introduced by Scheffer \cite{Scheffer}. It seems that the approach in \cite{BGS} is special for the Cauchy distributions and can not be applied for more general convex measures. For $p =1$ we obtain the sharp weighted Poincar\'e type inequality for Cauchy measures which was previously studied by Blanchet, Bonforte, Dolbeault, Grillo and Vazquez \cite{Blanchet1,Bonforte} with applications to the asymptotics of the fast diffusion equations \cite{Blanchet2,Bonforte} (see also \cite{BobkovLedoux,BJM16a,ABJ,VHN}): let $\beta \geq n+1$, then it holds
\[
\int_{\R^n} f^2 d\mu_\beta - \lt(\int_{\R^n} f d\mu_\beta\rt)^2 \leq \frac{1}{2 (\beta -1)}\int_{\R^n} |\na f|^2 (1+ |x|^2) d\mu_\beta
\]
for any smooth function $f$. It is remarkable that the constant $C_p = \frac{1}{2 (\beta -1)}$ in \eqref{eq:BecknerCauchy} is sharp in the sense that it can not be replaced by any smaller constant.
To see this, let $B_p$ denote the sharp constant in \eqref{eq:BecknerCauchy}, then obviously $B_p \leq \frac1{2(\beta-1)}$. For any smooth bounded function $g$ such that $\int_{\R^n} g d\mu_\beta =0$, applying \eqref{eq:BecknerCauchy} for $1 + \ep g$ with $\ep >0$ small enough and expanding the obtained inequality in term $\ep^2$, we get
\[
\ep^2 \int_{\R^n} g^2 d\mu_\beta + o(\ep^2) \leq B_p \ep^2 \int_{\R^n} |\na g|^2 \vphi d\mu_\beta,
\]
for $\ep >0$ small enough. Letting $\ep \to 0$ we have
\[
\int_{\R^n} g^2 d\mu_\beta \leq B_p \int_{\R^n} |\na g|^2 \vphi d\mu_\beta
\]
for any bounded smooth function $g$ with $\int_{\R^n} g d\mu_\beta =0$. This implies $B_p \geq B_1 = \frac1{2(\beta -1)}$. Consequently, we get $B_p = \frac1{2(\beta -1)}$.


The last remark concerning to Corollary \ref{Becknertypeineq} is that $p_\beta < 2$, hence we can not let $p \uparrow 2$ to obtain a weighted logarithmic Sobolev inequality for the convex measures $\mvb$ (or Cauchy measure $\mu_\beta$) with weighted $\vphi$. It's was shown in \cite{BobkovLedoux} that the weighted logarithmic Sobolev inequality for the Cauchy measures holds true with the weight $w(x) = (1 + |x|^2)^2 \ln (e + |x|^2)$. In \cite{CGW}, by using Lyapunov method, Cattiaux, Guillin and Wu found the correct order of magnitude of the weight in this inequality as $w(x) = (1+ |x|^2) \ln (e + |x|^2)$. Finally, we have $p_\beta \to 2$ as $\beta \to \infty$, we can see that the logarithmic Sobolev inequality for the uniform log-concave measure can be obtained from \eqref{eq:Beckner}. Indeed, suppose $d\mu = e^{-\psi} dx$ is a log-concave probability measure such that $D^2\psi \geq \rho \I_n$ for some $\rho >0$. For each $\beta > n+1$, consider the function $\vphi_\beta = 1 + \frac{\psi}\beta$ and the probability measure $\mu_{\vphi_\beta,\beta}$. We have $D^2 \vphi_\beta \geq c_\beta: = \frac{2\rho}\beta$. For any positive smooth function $f$, we apply \eqref{eq:Beckner} for $\mu_{\vphi_\beta,\beta}$, $f$ and $p =p_\beta$ and then let $\beta \to \infty$ with remark that $Z_{\vphi_\beta,\beta} \vphi_\beta^{-\beta} \to e^{-\psi}$ to obtain the following inequality
\[
\int_{\R^n} f^2 \ln f^2 d\mu -\int_{\R^n} f^2 d\mu \, \ln\lt(\int_{\R^n} f^2 d\mu\rt) \leq \frac{2}{\rho} \int_{\R^n} |\na f|^2 d\mu.
\]
Especially, when $\psi(x) = |x|^2/2$ we obtain the famous Gross's logarithmic--Sobolev inequality for Gaussian \cite{Gross}.

The second main result of this paper is the asymmetric covariance estimates for the convex measure $\mvb$. Let $\mu$ be a probability measure in $\R^n$. For any two real-valued function $g, h\in L^2(\mu)$, the covariance of $g$ and $h$ is quantity
\[
\text{\rm cov}_{\mu}(g,h) = \int_{\R^n} g h d\mu - \lt(\int_{\R^n} g d\mu\rt) \lt(\int_{\R^n} h d\mu\rt).
\]
Notice that $\text{\rm cov}_{\mu}(g,g) = \var_{\mu}(g)$. If $\mu$ is a log-concave measure, i.e., $d\mu = e^{-V(x)} dx$ for some strictly convex function $V$ on $\R^n$, the Brascamp--Lieb inequality (see \cite{BL}) asserts that
\begin{equation}\label{eq:BLinequality}
\var_\mu(h) \leq \int_{\R^n} \la (D^2 V)^{-1} \na h, \na h\ra d\mu,\quad h \in L^2(\mu).
\end{equation}
Since $(\text{\rm cov}_\mu(g,h))^2 \leq \var_\mu(g) \var_\mu(h)$, as an immediate consequence of \eqref{eq:BLinequality}, we have the following covariance estimate
\begin{equation}\label{eq:covestimate}
(\text{\rm cov}_\mu(g,h))^2 \leq \int_{\R^n} \la (D^2 V)^{-1} \na g, \na g\ra d\mu\, \int_{\R^n} \la (D^2 V)^{-1} \na h, \na h\ra d\mu.
\end{equation}
The one-dimensional variant of \eqref{eq:covestimate} was established by Menz and Otto \cite{MO} as follows
\begin{equation}\label{eq:MO}
|\text{\rm cov}_\mu(g,h)| \leq \|g'\|_{L^1(\mu)} \|(V'')^{-1} h'\|_{L^\infty(\mu)} = \int_\R |g'| d\mu\, \sup_{x\in \R} \frac{|h'(x)|}{V''(x)}.
\end{equation}
They call this inequality an asymmetric Brascamp--Lieb inequality. Note that it is asymmetric in
two respects: One respect is to take an $L^1$ norm of $g'$ and an $L^\infty$ norm of $h'$, instead of $L^2$ norm and $L^2$ norm. The second respect is that the $L^\infty$ norm is weighted with $(V''(x))^{-1}$ while the $L^1$ norm is not weighted.

The higher dimension version of \eqref{eq:MO} was proved by Carlen, Cordero--Erausquin and Lieb \cite{CCL}. In fact, they established a more general estimate as follows: let $\lam_{min}(x)$ denotes the smallest eigenvalued of $D^2V(x)$ then for any (locally) Lipschitz functions $f, g \in L^2(\mu)$ and for any $2 \leq p \leq \infty$ and $q =p/(p-1)$ we have
\begin{equation}\label{eq:CCL}
|\text{\rm cov}_\mu(g,h)| \leq \|(D^2 V)^{-\frac1q} \na g\|_{L^q(\mu)} \|\lam_{min}^{\frac{2-p}p} (D^2 V)^{-\frac1p} \na h\|_{L^p(\mu)}.
\end{equation}
The inequality \eqref{eq:CCL} is sharp in the sense that the constant $1$ in the right hand side can not be replaced by any smaller constant. For $p=2$ we recover \eqref{eq:covestimate} from \eqref{eq:CCL}. Since $D^2 V \geq \lam_{min} \I_n$ then \eqref{eq:CCL} implies
\[
|\text{\rm cov}_\mu(g,h)| \leq \|\lam_{min}^{-\frac1p} \na g\|_{L^p(\mu)} \|\lam_{min}^{-\frac1q} \na h\|_{L^q(\mu)}.
\]
For $p =\infty$ and $q =1$, we get
\[
|\text{\rm cov}_\mu(g,h)| \leq \|\na g\|_{L^\infty(\mu)} \|\lam_{min}^{-1} \na h\|_{L^1(\mu)}.
\]
In particular, if $n=1$ we obtain the inequality \eqref{eq:MO} of Menz and Otto.

In this paper, we extend the asymmetric covariance estimate \eqref{eq:CCL} to the convex measure $\mvb$. For $n \geq 1$ and $\beta \geq n+1$, let us denote
\[
p_{\be,n} =
\begin{cases}
\infty &\mbox{if $n =1$,}\\
2(1 + \frac{(\beta -1) (\beta -n-1) +((\be -1)(\be -2) (\be -n)(\be -n-1))^{\frac12}}{n-1})&\mbox{if $n\geq 2$.}
\end{cases}
\]
Our next result is the following theorem.
\begin{theorem}\label{asymmetriccov}
Let $\beta \geq n+1$ and $\lam_{min}$ denotes the smallest eigenvalue of $D^2\vphi(x)$. Then for any $2\leq p \leq p_{\beta,n}$, $q =p/(p-1)$ and any (locally) Lipschitz functions $g,h$ in $L^2(\mvb)$, we have
\begin{equation}\label{eq:covarianceconvex}
|\text{\rm cov}_{\mvb}(g,h)|\leq \frac1{\beta -1} \lt(\int_{\R^n} |(D^2\vphi)^{-\frac 1p} \na g|^q \vphi d\mvb\rt)^{\frac1q} \lt(\int_{\R^n} \lam_{min}^{2-p} |(D^2\vphi)^{-\frac1p} \na h|^p \vphi d\mvb\rt)^{\frac1p}.
\end{equation}
\end{theorem}
It is interesting that Theorem \ref{asymmetriccov} implies the asymmetric covariance estimates \eqref{eq:CCL} of Carlen, Cordero-Erausquin and Lieb for log-concave measure by letting $\beta \to \infty$. We will show this fact in Sect. \ref{Asym} below.

We conclude this introduction by giving some comments on the methods used to prove our Theorem \ref{Phientropy} and Theorem \ref{asymmetriccov}. Theorem \ref{Phientropy} is proved by using the semi-group method while Theorem \ref{asymmetriccov} is proved by adapting the $L^2-$method of H\"ormander \cite{Hormander} to the $L^p$ setting. Both the proofs concern to a differential operator $L$ on $L^2(\mvb)$ defined by
\[
L f(x) = \vphi(x) \Delta f(x) -(\beta -1) \la \na \vphi(x), \na f(x)\ra.
\]
To prove Theorem \ref{Phientropy}, we consider the semi-group $P_t$  on $L^2(\mvb)$ associated with $L$, and define the function
\[
\al(t) = -\int_{\R^n} \Phi(P_t f) d\mvb, \quad f \in L^2(\mvb).
\]
Using the semi-group property of $P_t$ and the assumption on $\Phi$, we will establish the following differential inequality $\alpha''(t) \leq -2c(\beta -1) \alpha'(t),\quad t > 0$,
which leads to the $\Phi-$entropy inequalities. We notice that the semi-group method is an useful methods to prove the functional inequalities (especially in sharp form). We refer the readers to the paper \cite{BolleyGentil,BGL,BE} and references therein for more details about this method an its applications. The $L^2-$approach of H\"ormander \cite{Hormander} is based on the classical dual representation for the covariance to establish the spectral estimates. In \cite{CCL}, Carlen, Cordero--Erausquin and Lieb adapted the $L^2$ approach of H\"ormander to the $L^p$ setting to prove the inequality \eqref{eq:CCL} for log-concave measure. Our proof of Theorem \ref{asymmetriccov} is an adaptation of their method to the setting of convex measures. However, the computations in our situation are more complicated. 

The rest of this paper is organized as follows. In Sect. \ref{provePhientropy} we use the semi-group method to prove the $\Phi-$entropy inequality in Theorem \ref{Phientropy} and show how derive the $\Phi-$entropy inequalities for uniform log-concave measures from Theorem \ref{Phientropy}. Sect. \ref{Asym} is devoted to prove the asymmetric covariance estimates for convex measures in Theorem \ref{asymmetriccov} and show how derive the inequality of Carlen, Cordero--Erausquin and Lieb from this theorem.

\section{Proof of Theorem \ref{Phientropy}}\label{provePhientropy}
This section is devoted to prove Theorem \ref{Phientropy}. Assume that $D^2\vphi \geq c \I_n$ for some $c >0$ and $\beta > n+1$. As in the introduction, let us define a differential operator $L$ of order $2$ on  $C_c^\infty(\R^n)$ by
\[
L f(x) = \vphi(x) \Delta f(x) - (\beta-1) \la \na\vphi(x), \na f(x)\ra, \quad f\in C_c^\infty(\R^n).
\]
By integration by parts, we have
\[
\int_{\R^n} (Lf)\, g\, d\mvb = - \int_{\R^n} \la \na f, \na g\ra \vphi d\mvb, \quad f, g \in C_c^\infty(\R^n).
\]
Since $D^2\vphi(x) \geq c \I_n$, $c >0$ then the following weighted Poincar\'e inequality holds (see \cite{VHN}):
\[
\var_{\mvb}(f) \leq \frac1{2c(\beta -1)} \int_{\R^n} |\na f|^2 \vphi(x) d\mvb, \quad f \in C_c^\infty(\R^n).
\]
Hence the operator $L$ is uniquely extended to a self-adjoint operator on $L^2(\mvb)$ (we still denoted the extended operator by $L$) with domain $\mathcal D(L)$. Notice that $C_c^\infty$ is dense in $\mathcal D(L)$ under the norm $(\|f\|_{L^2(\mvb)}^2 + \|Lf\|_{L^2(\mvb)}^2)^{\frac12}$. Let $P_t$ denote the semi-group on $L^2(\mvb)$ generated by $L$. For any $f \in L^2(\mvb)$ then $P_t f\in \mathcal D(L)$ and satisfies the equation
\[
\frac{\pa P_tf}{\pa t}(x) = L P_t f(x),\quad P_0f(x) = f(x).
\]
Moreover, $P_t f \to \int_{\R^n} f d\mvb$  in $L^2(\mvb)$ and $\mvb-$a.e. in $\R^n$ as $t\to \infty$. With these preparations, we are now ready to prove Theorem \ref{Phientropy}.

\begin{proof}[Proof of Theorem \ref{Phientropy}]
Let $f \in L^2(\mvb)$ such that $\int_{\R^n} |\na f|^2 \vphi(x) d\mvb < \infty$. Define the function
\[
\alpha(t) = -\int_{\R^n} \Phi(P_t(f)) d\mvb.
\]
By integration by parts, we have the following expression for $\alpha'(t)$
\begin{equation}\label{eq:alpha1st}
\al'(t) = -\int_{\R^n} \Phi'(P_t f) \, L P_t f d\mvb = \int_{\R^n} \Phi''(P_t f) |\na P_t f|^2 \vphi\, d\mvb.
\end{equation}
We next compute $\alpha''(t)$. For simplicity, we denote $g = P_tf$. It is easily to verify the following relation
\begin{equation}\label{eq:Commutation}
\pa_i (L g) = L (\pa_i g) + \pa_i \vphi \Delta g - (\beta-1) \sum_{j=1}^n \pa^2_{ij}\vphi \pa_j g,\quad i =1, 2, \ldots,n,
\end{equation}
where $\pa_i = \frac{\pa}{\pa x_i}$ and $\pa^2_{ij} = \frac{\pa^2}{\pa x_i \pa x_j}$. Using the relation \eqref{eq:Commutation} and integration by parts, we have
\begin{align}\label{eq:alpha2nd}
\al''(t) &= \int_{\R^n} \Phi^{(3)}(g) |\na g|^2 L g\, \vphi d\mvb + 2\int_{\R^n} \Phi''(g) \la \na g, \na Lg\ra\, \vphi d\mvb\notag\\
&=-\int_{\R^n} \la \na(\Phi^{(3)}(g) |\na g|^2 \vphi), \na g\ra \vphi d\mvb + 2 \int_{\R^n} \Phi''(g) \la \na g, L(\na g)\ra \, \vphi d\mvb\notag\\
&\quad + 2\int_{\R^n} \Phi''(g) \la \na g, \na \vphi\ra \Delta g\, \vphi d\mvb - 2(\beta -1)\int_{\R^n} \Phi''(g) \la D^2\vphi \na g, \na g\ra\, \vphi d\mvb,
\end{align}
here, for simplifying notation, we denote $L(\na g) = (L(\pa_1 g), \ldots, L(\pa_n g))$. It follows from intgeration by parts that
\begin{align}\label{eq:IBP0000}
\int_{\R^n} \Phi''(g)& \la \na g, L(\na g)\ra \, \vphi d\mvb\notag\\
& =-\sum_{i=1}^n\int_{\R^n} \la \na \pa_i g, \na(\Phi''(g) \pa_i g \vphi)\ra \vphi d\mvb\notag\\
&= -\int_{\R^n} \Phi^{(3)}(g)\la \na^2 g \na g, \na g\ra \vphi^2 d\mvb -\int_{\R^n} \Phi''(g) \|\na^2 g\|_{HS}^2 \vphi^2 d\mvb\notag\\
&\quad -\int_{\R^n} \Phi''(g) \la \na^2 g\na g, \na \vphi\ra \vphi d\mvb.
\end{align}
Noting that
\begin{equation}\label{eq:RELATION}
\na(\Phi^{(3)}(g) |\na g|^2 \vphi) = \Phi^{(4)}(g) |\na g|^2 \vphi \na g+ 2\Phi^{(3)}(g) \vphi D^2 g \na g + \Phi^{(3)}(g) |\na g|^2 \na \vphi.
\end{equation}
Plugging \eqref{eq:IBP0000} and \eqref{eq:RELATION} into \eqref{eq:alpha2nd} and using the uniform convexity assumption $D^2\vphi \geq c \I_n, c >0$ of $\vphi$ we obtain
\begin{align}\label{eq:alpha2nda}
\alpha''(t) 
&\leq -2c(\beta-1) \alpha'(t) -\int_{\R^n} \Phi^{(4)}(g) |\na g|^4 \vphi^2 d\mvb -4 \int_{\R^n} \Phi^{(3)}(g) \la D^2g \na g, \na g\ra \vphi^2 d\mvb\notag\\
&\quad -\int_{\R^n} \Phi^{(3)}(g)|\na g|^2 \la \na \vphi, \na g\ra \vphi d\mvb -2 \int_{\R^n} \Phi''(g) \|D^2 g\|_{HS}^2 \vphi^2 d\mvb\notag\\
&\quad -2 \int_{\R^n} \Phi''(g) \la D^2 g \na g, \na \vphi\ra \vphi d\mvb + 2\int_{\R^n} \Phi''(g) \la \na g, \na \vphi\ra \Delta g\, \vphi d\mvb.
\end{align}
Using again integration by parts, we have
\begin{align}\label{eq:IBP1}
\int_{\R^n} \Phi^{(3)}(g)&|\na g|^2 \la \na \vphi, \na g\ra \vphi d\mvb\notag\\
& = -\frac1{\beta -2} \int_{\R^n} \Phi^{(3)}(g)|\na g|^2 \la \na g, \na \vphi^{-\beta +2}\ra \frac{dx}{Z_{\vphi,\beta}}\notag\\
&= \frac1{\beta -2} \int_{\R^n} \Phi^{(4)}(g)|\na g|^4 \vphi^2 d\mvb + \frac{2}{\beta -2}\int_{\R^n} \Phi^{(3)}(g) \la D^2g \na g, \na g\ra \vphi^2 d\mvb\notag\\
&\quad  +\frac{1}{\beta -2}\int_{\R^n} \Phi^{(3)}(g) |\na g|^2 \Delta g\, \vphi^2 d\mvb,
\end{align}
\begin{align}\label{eq:IBP2}
&\int_{\R^n} \Phi''(g) \la D^2 g \na g, \na \vphi\ra \vphi d\mvb\notag\\
&\qquad\qquad= -\frac{1}{\beta -2}\int_{\R^n} \Phi''(g) \la D^2 g \na g, \na \vphi^{-\beta +2}\ra \frac{dx}{Z_{\vphi,\beta}}\notag\\
&\qquad\qquad= \frac{1}{\beta -2}\int_{\R^n} \Phi^{(3)}(g) \la D^2 g \na g, \na g\ra \vphi^2 d\mvb + \frac1{\beta -2} \int_{\R^n} \Phi''(g) \la \na \Delta g, \na g\ra \vphi^2 d\mvb\notag\\
&\qquad\qquad \quad + \frac1{\beta -2} \int_{\R^n} \Phi''(g) \|D^2 g\|_{HS}^2 \vphi^2 d\mvb,
\end{align}
and
\begin{align}\label{eq:IBP3}
\int_{\R^n} \Phi''(g) &\la \na g, \na \vphi\ra \Delta g\, \vphi d\mvb\notag\\
& = -\frac1{\beta -2} \int_{\R^n} \Phi''(g) \Delta g\, \la \na g, \na \vphi^{-\beta +2}\ra  \frac{dx}{Z_{\vphi,\beta}}\notag\\
&= \frac1{\beta -2} \int_{\R^n} \Phi^{(3)}(g)|\na g|^2 \De g\, \vphi^2 d\mvb + \frac1{\beta -2} \int_{\R^n} \Phi''(g) \la \na \Delta g, \na g\ra \vphi^2 d\mvb\notag\\
&\quad + \frac1{\beta -2} \int_{\R^n} \Phi''(g) (\Delta g)^2 \, \vphi^2 d\mvb.
\end{align}
Inserting \eqref{eq:IBP1}, \eqref{eq:IBP2} and \eqref{eq:IBP3} into \eqref{eq:alpha2nda}, we get
\begin{align}\label{eq:alpha2ndb}
\alpha''(t) &\leq -2c(\beta -1) \alpha'(t) -\frac{\beta -1}{\beta -2} \int_{\R^n} \Phi^{(4)}(g)|\na g|^4 \vphi^2 d\mvb\notag\\
&\quad -\int_{\R^n} \Phi^{(3)}(g) \lt(\frac{4(\beta -1)}{\beta -2} \la D^2 g\na g,\na g\ra - \frac1{\beta -2} |\na g|^2 \De g\rt) \vphi^2 d\mvb\notag\\
&\quad - 2\int_{\R^n} \Phi''(g) \lt(\frac{\beta -1}{\beta -2} \|D^2 g\|_{HS}^2 - \frac{1}{\beta -2} (\De g)^2\rt) \vphi^2 d\mvb.
\end{align}
It is well known that $(\De g)^2 \leq n \|D^2 g\|_{HS}^2$, then it holds
\begin{equation}\label{eq:matrixineq1}
\frac{\beta -1}{\beta -2} \|D^2 g\|_{HS}^2 - \frac{1}{\beta -2} (\De g)^2 \geq \frac{\beta -n -1}{\beta -2}\|D^2 g\|_{HS}^2.
\end{equation}
Let $\lam_1, \ldots, \lam_n$ denote the eigenvalue of $D^2 g$ with respect to the eigenvector $e_1, \ldots, e_n$ respectively such that $|e_i| =1$ for any $i =1,2,\ldots,n$. Denote $a_i =\frac{\la \na g,e_i\ra^2}{|\na g|^2}$ then it holds $a_1 + \cdots + a_n =1$, $a_i \geq 0$ for $i= 1,\ldots,n$. Using these notation, we have
\begin{align*}
\frac{4(\beta -1)}{\beta -2} \la D^2 g\na g,\na g\ra - \frac1{\beta -2} |\na g|^2 \De g &= |\na g|^2 \lt(\frac{4(\beta -1)}{\beta -2} \sum_{i=1}^n \lam_i a_i - \frac1{\beta -2} \sum_{i=1}^n \lam_i\rt)\\
&= |\na g|^2 \sum_{i=1}^n \frac{4(\beta -1) a_i -1}{\beta -2} \lam _i.
\end{align*}
Using Cauchy-Schwartz inequality, we have
\begin{align*}
\lt(\sum_{i=1}^n \frac{4(\beta -1) a_i -1}{\beta -2} \lam _i\rt)^2 &\leq \lt(\sum_{i=1}^n \lt(\frac{4(\beta -1)a_i -1}{\beta -2}\rt)^2\rt) (\lam_1^2 + \cdots + \lam_n^2)\\
&= \frac{(4(\beta -1))^2 \sum_{i=1}^n a_i^2 - 8(\beta -1) + n}{(\beta -2)^2} \|D^2 g\|_{HS}^2\\
&\leq \frac{16(\beta -1)^2 - 8(\beta -1) + n}{(\beta -2)^2} \|D^2 g\|_{HS}^2,
\end{align*}
here we used $\sum_{i=1}^n a_i =1$, $\sum_{i=1}^n a_i^2 \leq 1$ and $\|D^2 g\|_{HS}^2 = \sum_{i=1}^n \lam_i^2$. Putting the previous estimates together, we get
\begin{equation}\label{eq:matrixineq2}
\lt|\frac{4(\beta -1)}{\beta -2} \la D^2 g\na g,\na g\ra - \frac{|\na g|^2 \De g}{\beta -2} \rt| \leq \frac{((4\beta -5)^2 + n -1)^{\frac12}}{\beta -2} \|D^2 g\|_{HS} |\na g|^2.
\end{equation}
Plugging \eqref{eq:matrixineq1} and \eqref{eq:matrixineq2} into \eqref{eq:alpha2ndb} and using $\Phi'' \geq 0$, we obtain
\begin{align*}
\al''(t)& \leq -2c(\beta -1)\al'(t) -\frac{\beta -1}{\beta -2} \int_{\R^n} \Phi^{(4)}(g)|\na g|^4 \vphi^2 d\mvb\notag\\
&\quad + \frac{((4\beta -5)^2 + n -1)^{\frac12}}{\beta -2} \int_{\R^n} |\Phi^{(3)}(g)|\, \|D^2 g\|_{HS} |\na g|^2 \vphi^2 d\mvb\\
&\quad -2 \frac{\beta -n -1}{\beta -2} \int_{\R^n} \Phi''(g) \|D^2 g\|_{HS}^2 \vphi^2 d\mvb.
\end{align*}
It follows from the assumption on $\Phi$ and Cauchy--Schwartz inequality that
\begin{align*}
\frac{\beta-1}{\beta-2} \Phi^{(4)}(g)|\na g|^4 & + 2 \frac{\beta -n -1}{\beta -2} \Phi''(g) \|D^2 g\|_{HS}^2\\
&\geq 2\frac{\sqrt{2(\beta-1)(\beta-n-1)\Phi^{(4)}(g) \Phi''(g)}}{\beta -2} |\na g|^2 \|D^2 g\|_{HS}\\
&\geq \frac{((4\beta -5)^2 + n-1)^{\frac12}}{\beta -2} |\Phi^{(3)}(g)||\na g|^2 \|D^2 g\|_{HS}
\end{align*}
Therefore, it is easy to check that
\[
\alpha''(t) \leq -2c(\beta -1) \alpha'(t), \quad t >0.
\]
This differential inequality implies $\alpha'(t) \leq e^{-2c(\beta -1) t} \al'(0)$. Integrating the latter inequality on $(0, \infty)$ we obtain
\[
\lim_{t\to \infty} \alpha(t) - \alpha(0) \leq \frac{1}{2c(\beta -1)} \alpha'(0)
\]
which yields the $\Phi-$entropy inequality \eqref{eq:Phientropyineq} because 
\[
\alpha(0) = -\int_{\R^n} \Phi(f) d\mvb,\quad \alpha'(0) = \int_{\R^n} \Phi''(f) |\na f|^2 \vphi d\mvb,
\]
and
\[
\lim_{t\to \infty} \alpha(t)  = -\Phi\lt(\int_{\R^n} f d\mvb\rt)
\]
since $P_t f \to \int_{\R^n} f d\mvb$ in $L^2(\mvb)$. The proof of Theorem \ref{Phientropy} is then completely finished.
\end{proof}
We conclude this section by showing that the $\Phi-$entropy inequality \eqref{eq:PhientropyBG} can be derived from our Theorem \ref{Phientropy}. Let $\psi$ be a convex function on $\R^n$ such that $D^2\psi \geq \rho \I_n$ for some $\rho >0$ and $\int_{\R^n} e^{-\psi} dx =1$. Denote $\mu$ the measure on $\R^n$ with density $e^{-\psi}$. For $\beta > n+1$, denote $\vphi_\beta = 1 + \frac{\psi}{\beta}$. By the uniform convexity of $\psi$, we have $\vphi_\beta >0$ on $\R^n$ for $\beta$ large enough and $D^2\vphi_\beta \geq \be^{-1} \rho \I_n$. Denote $Z_{\vphi_\beta,\beta} = \int_{\R^n} \psi_\beta^{-\beta} dx$ and $\mu_{\vphi_\beta, \beta}$ the probability measure with density $Z_{\vphi_\beta,\beta}^{-1} \vphi_\beta^{-\beta}$. Our aim is to apply the $\Phi-$entropy inequality \eqref{eq:Phientropyineq} for the measure $\mu_{\varphi_\beta,\beta}$ and then letting $\beta \to \infty$ to derive the inequality \eqref{eq:PhientropyBG}. However, there is a difficulty here that although
\[
\lim_{\beta\to \infty} \frac1 8 \frac{(4\beta -5)^2 + n-1}{(\beta-1) (\beta -n -1)} =2,
\]
but
\[
\frac1 8 \frac{(4\beta -5)^2 + n-1}{(\beta-1) (\beta -n -1)} > 2,
\]
for any $\beta > n+1$. Hence for a convex function $\Phi$ satisfying $\Phi'' \Phi^{(4)} \geq 2 (\Phi^{(3)})^2$ we do not know whether or not it satisfies \eqref{eq:Phicondition}. To overcome this difficulty, we use a approximation process as follows. Denote by $I$ the domain of $\Phi$. Let $I_0 =(a,b)$ be a bounded interval in $I$ such that $\bar {I}_0 \subset I$. Denote $M =\sup_{I_0} |\Phi^{(3)}| < \infty$. Notice that the function $\Psi_p(t) = (t-a +1)^p$ for $p\in (1,2)$ satisfies 
\[
\Psi_p'' \Psi_p^{(4)} = \frac{3-p}{2-p} (\Psi_p^{(3)})^2 = \ga_p (\Psi_p^{(3)})^2, \quad \ga_p = \frac{3-p}{2-p} > 2.
\]
For $\ep >0$, consider the function $\Phi_\ep = \Phi + \ep \Psi_p$ on $I_0$. By Cauchy-Schwartz inequality, we have
\[
\Phi_\ep'' \Phi_\ep^{(4)} \geq (\sqrt{2} |\Phi^{(3)}| + \sqrt{\ga_p} \ep |\Psi_p^{(3)}|)^2
\]
on $I_0$. Denote $N = \inf_{I_0} |\Psi_p^{(3)}| > 0$. It is easy to check that
\[
(\sqrt{2} |\Phi^{(3)}| + \sqrt{\ga_p} \ep |\Psi_p^{(3)}|)^2 \geq \de ( |\Phi^{(3)}| + \ep |\Psi_p^{(3)}|)^2,
\]
on $I_0$, for any 
\[
2 < \de < \min\lt\{\sqrt{2\ga_p}, \frac{2M^2 + \ga_p \ep^2 N^2}{M^2 + \ep^2 N^2}\rt\}.
\]
Consequently, the function $\Phi_\ep$ satisfies the condition \eqref{eq:Phicondition} on $I_0$ for $\beta >0$ large enough. Applying the inequality \eqref{eq:Phientropyineq} for the convex function $\Phi_\ep$ and for any smooth function $f$ with value in $I_0$ and the probability measure $\mu_{\vphi_\beta,\beta}$ with $\beta$ large enough, we have
\[
\int_{\R^n} \Phi_\ep(f) d\mu_{\varphi_\beta} - \Phi_{\ep}\lt(\int_{\R^n} f d\mu_{\varphi_\beta}\rt) \leq \frac1{2 \frac{\rho}{\beta}(\beta -1)} \int_{\R^n} \Phi_\ep''(f) |\na f|^2 \vphi_\beta d\mu_{\vphi_\beta,\beta}.
\]
Notice that $Z_{\vphi_\beta,\beta}^{-1} \vphi_\beta^{-\beta} \to e^{-\psi}$ and $\vphi_\beta \to 1$. Letting $\beta \to \infty$ and then letting $\ep \to 0$, we get
\begin{equation}\label{eq:bound}
\int_{\R^n} \Phi(f) d\mu  - \Phi\lt(\int_{\R^n} f d\mu\rt) \leq \frac1{2\rho} \int_{\R^n} \Phi''(f) |\na f|^2 d\mu,
\end{equation}
for any smooth function $f$ with value in $I_0$ and for any bounded interval $I_0 \subset I$ with $\bar I_0 \subset I$. Suppose $I = (a,b)$, let $(a_n)_n, (b_n)_n$ be two sequence such that $a_n \downarrow a$ and $b_n\uparrow b$. For any smooth function $f$ with value in $I$, define $f_n = \max\{a_n, \min\{f, b_n\}\}$. Applying the inequality \eqref{eq:bound} for $I_n$ and $f_n$ and then letting $n \to \infty$ we obtain the inequality \eqref{eq:PhientropyBG} for $f$.

\section{Proof of Theorem \ref{asymmetriccov}}\label{Asym}
In this section, we prove the asymmetric covariance estimates given in Theorem \ref{asymmetriccov}. Our method is based on the $L^2$ method of H\"ormander which turns out to be very useful to prove the Brascamp--Lieb type and Poincar\'e type inequalities (see, e.g., \cite{VHN,CCL}). Again, let $L$ denote the differential operator 
\[
Lf(x) = \vphi(x) \Delta f(x) - (\beta -1) \la \na \vphi(x), \na f(x)\ra, \quad f\in C_c^\infty(\R^n).
\]
Note by integration by parts that
\[
\int_{\R^n}  g Lf d\mvb = -\int_{\R^n} \la \na g, \na f\ra \vphi d\mu,\quad f, g \in C_c^\infty(\R^n).
\]
hence $L$ is extended uniquely to self-adjoint operator in $L^2(\mvb)$ (which we still denote by $L$). By approximation argument, we can assume that $\vphi$ is uniform convex in $\R^n$. Consequently, if we denote $P_t$ the semi-group associated with $L$, then by the weighted Poincar\'e inequality, we see that $\|P_t h\|_{L^2_{\mvb}}$ exponentially decays to $0$ for any function $h\in L^2(\mvb)$ with $\int_{\R^n} h \,d\mvb =0$. For such a function $h$, the integral
\begin{equation}\label{eq:Bochnertp}
u := \int_0^\infty P_t h dt,
\end{equation}
exists and is in the domain of $L$, and satisfies $Lu =h$. 

Since
\[
\text{\rm cov}_{\mvb}(g,h) = \int_{\R^n} g(x) \lt(h(x) -\int_{\R^n} h d\mvb\rt) d\mvb,
\]
then $\text{\rm cov}_{\mvb}(g,h+c) = \text{\rm cov}_{\mvb}(g,h)$ for any constant $c$. Whence we can assume that $\int_{\R^n} h d\mvb =0$. Let $u$ define by \eqref{eq:Bochnertp}. We have by integration by parts and approximation argument that
\begin{equation}\label{eq:IBP*}
\text{\rm cov}_{\mvb}(g,h) = \int_{\R^n} g(x) h(x) d\mvb  = \int_{\R^n} g(x) L u(x) d\mvb =-\int_{\R^n} \la \na g, \na u\ra \vphi d\mvb.
\end{equation}
With these preparations, we are now ready to give the proof of Theorem \ref{asymmetriccov}.

\begin{proof}[Proof of Theorem \ref{asymmetriccov}]
We can assume $\int_{\R^n} h d\mvb =0$. Let $u$ define by \eqref{eq:Bochnertp}. Using \eqref{eq:IBP*} and  H\"older inequality, we have
\begin{align}\label{eq:Holder}
|\text{\rm cov}_{\mvb}(g,h)| & = \lt|\int_{\R^n} \la \na g, \na u\ra \vphi d\mvb\rt| = \lt|\int_{\R^n} \la (D^2\vphi)^{-\frac1p} \na g, (D^2\vphi)^{\frac1p} \na u\ra \vphi d\mvb\rt|\notag\\
&\leq \lt(\int_{\R^n} |(D^2\vphi)^{-\frac1p} \na g|^q \vphi d\mvb\rt)^{\frac1q} \lt(\int_{\R^n} |(D^2\vphi)^{\frac1p} \na u|^p \vphi d\mvb\rt)^{\frac1p}, 
\end{align}
here recall $q = p/(p-1)$. It remains to show that 
\begin{equation}\label{eq:suffice}
\lt(\int_{\R^n} |(D^2\vphi)^{\frac1p} \na u|^p \vphi d\mvb\rt)^{\frac1p} \leq \frac1{\beta -1} \lt(\int_{\R^n} \lam_{min}^{2-p} |(D^2\vphi)^{-\frac1p} \na h|^p \vphi d\mvb\rt)^{\frac1p},
\end{equation}
where $\lam_{min}$ is the smallest eigenvalue of $D^2\vphi$. To prove \eqref{eq:suffice}, we first compute $L(|\na u|^p)$ as follows
\begin{align}\label{eq:ex1}
L(|\na u|^p) & = \vphi \Delta (|\na u|^p) - (\beta -1) \la \na \vphi, \na (|\na u|^p)\ra\notag\\
&=p \vphi |\na u|^{p-2} \|D^2 u\|_{HS}^2 + p |\na u|^{p-2}\sum_{j=1}^n \vphi \Delta (\pa_j u) \pa_j u  + p(p-2) \vphi |\na u|^{p-4} |D^2 u\na u|^2\notag\\
&\quad - p(\beta -1) |\na u|^{p-2} \sum_{j=1}^n \lt(\sum_{i=1}^n \pa_i \vphi \pa_{ij}^2 u\rt) \pa_j u\notag\\
&= p |\na u|^{p-2} \lt(\la L(\na u), \na u\ra + \vphi \|D^2 u\|_{HS}^2 + (p-2) \vphi \frac{ |D^2 u\na u|^2}{|\na u|^2}\rt),
\end{align}
here we use the notation $L(\na u) = (L(\pa_1 u), \ldots, L(\pa_n u))$.

By integration by parts, we have
\begin{align}\label{eq:tpLnablaup*}
\int_{\R^n} L(|\na u|^p) \vphi d\mvb& = -\int_{\R^n} \la \na (|\na u|^p), \na \vphi\ra \varphi d\mvb  \notag\\
&= \frac1{\beta -2} \int_{\R^n} \la \na (|\na u|^p), \na \vphi^{-\beta +2}\ra \frac{dx}{Z_{\vphi,\beta}} \notag\\
&=-\frac{1}{\beta -2}\int_{\R^n} \Delta(|\na u|^p) \vphi^2 d\mvb.
\end{align}
We are readily to check that
\[
\De(|\na u|^p) = p\lt(\la \na \De u, \na u\ra + \|D^2 u\|_{HS}^2 + (p-2) \frac{|D^2u \na u|^2}{|\na u|^2}\rt) |\na u|^{p-2}.
\]
Plugging the previous identity into \eqref{eq:tpLnablaup*}, we arrive
\begin{align}\label{eq:tpLnablaup}
\int_{\R^n} &L(|\na u|^p) \vphi d\mvb\notag\\
&= -\frac{p}{\be -2} \int_{\R^n} \lt(\la \na \De u, \na u\ra + \|D^2 u\|_{HS}^2 + (p-2) \frac{|D^2 u \na u|^2}{|\na u|^2} \rt) |\na u|^{p-2}\vphi^2 d\mvb.
\end{align}
From \eqref{eq:Commutation}, we have 
\[
L(\na u) = \na (Lu) - \Delta u \, \na \vphi + (\beta -1) D^2\vphi \na u.
\]
Using this commutation relation together with \eqref{eq:ex1} and $Lu = h$, we get
\begin{align}\label{eq:tpLnablaup1}
\int_{\R^n}& L(|\na u|^p) \vphi d\mvb \notag\\
&= p\int_{\R^n} \la \na h, \na u\ra |\na u|^{p-2} \vphi d\mvb + p(\beta -1) \int_{\R^n} \la D^2 \vphi \na u, \na u\ra |\na u|^{p-2} \vphi d\mvb \notag\\
&\quad + p\int_{\R^n} \lt(\|D^2 u\|_{HS}^2 + (p-2) \frac{|D^2 u \na u|^2}{|\na u|^2} \rt) |\na u|^{p-2}\vphi^2 d\mvb\notag\\
&\quad -p \int_{\R^n} \Delta u\, |\na u|^{p-2} \la \na u, \na \vphi\ra \vphi d\mvb.
\end{align}
Using integration by parts, we have
\begin{align*}
\int_{\R^n} \Delta u\,& |\na u|^{p-2} \la \na u, \na \vphi\ra \vphi d\mvb \\
&= -\frac1{\beta -2} \int_{\R^n} \Delta u\, |\na u|^{p-2} \la \na u, \na \vphi^{-\beta +2}\ra \frac{dx}{Z_{\vphi,\beta}}\\
&=\frac1{\beta -2} \int_{\R^n} \lt(\la \na \De u, \na u\ra +(\De u)^2 + (p-2) \De u \frac{\la D^2 u\na u,\na u\ra}{|\na u|^2} \rt) |\na u|^{p-2} \vphi^2 d\mvb. 
\end{align*}
Inserting the previous equality into \eqref{eq:tpLnablaup1} implies 
\begin{align}\label{eq:tpLnablaup2}
&\int_{\R^n} L(|\na u|^p) \vphi d\mvb \notag\\
&\,= p\int_{\R^n} \la \na h, \na u\ra |\na u|^{p-2} \vphi d\mvb + p(\beta -1) \int_{\R^n} \la D^2 \vphi \na u, \na u\ra |\na u|^{p-2} \vphi d\mvb \notag\\
&\,\quad + p\int_{\R^n} \lt(\|D^2 u\|_{HS}^2 + (p-2) \frac{|D^2 u \na u|^2}{|\na u|^2} \rt) |\na u|^{p-2}\vphi^2 d\mvb\notag\\
&\,\quad -\frac p{\beta -2} \int_{\R^n} \lt(\la \na \De u, \na u\ra +(\De u)^2 + (p-2) \De u \frac{\la D^2 u\na u,\na u\ra}{|\na u|^2} \rt) |\na u|^{p-2} \vphi^2 d\mvb. 
\end{align}
Combining \eqref{eq:tpLnablaup} and \eqref{eq:tpLnablaup2}, we get
\begin{align}\label{eq:zero}
0& = p\int_{\R^n} \la \na h, \na u\ra |\na u|^{p-2} \vphi d\mvb + p(\beta -1) \int_{\R^n} \la D^2 \vphi \na u, \na u\ra |\na u|^{p-2} \vphi d\mvb \notag\\
&\quad +\frac{p}{\beta -2} \int_{\R^n} \lt( (\beta -1)\|D^2 u\|_{HS}^2  - (\De u)^2\rt)|\na u|^{p-2} \vphi^2 d\mvb \notag\\
&\quad + \frac{p(p-2)}{\beta -2} \int_{\R^n} \lt((\beta -1) \frac{|D^2 u \na u|^2}{|\na u|^2} - \Delta u \frac{\la D^2 u\na u,\na u\ra}{|\na u|^2} \rt)|\na u|^{p-2} \vphi^2 d\mvb.
\end{align}
We next claim that if $|\na u| > 0$ then
\begin{equation}\label{eq:claim}
 (\beta -1)\|D^2 u\|_{HS}^2  - (\De u)^2 + (p-2)\lt((\beta -1) \frac{|D^2 u \na u|^2}{|\na u|^2} - \Delta u \frac{\la D^2 u\na u,\na u\ra}{|\na u|^2} \rt) \geq 0
\end{equation}
provided $2 \leq p \leq p_{\beta,n}$. Indeed, if $n =1$ then the left hand side of \eqref{eq:claim} is equal to $(\beta -2) (p-1) |u''|^2$ and hence is non-negative. We next consider the case $n \geq 2$. Let $\lam_1, \ldots, \lam_n$ denote the eigenvalues of $D^2u$ with respect to the eigenvectors $e_1, \ldots, e_n$ respectively such that $|e_i| =1$ for any $i =1,\ldots,n$. Denote $a_i = \frac{\la \na u,e_i\ra^2}{|\na u|^2} \in [0,1]$. We have $a_1 + \cdots + a_n =1$, $\Delta u = \sum_{i=1}^n \lam_i$, $\|D^2 u\|_{HS}^2 = \sum_{i=1}^n \lam_i^2$, and
\[
\frac{|D^2 u \na u|^2}{|\na u|^2} = \sum_{i=1}^n \lam_i^2 a_i,\quad \frac{\la D^2 u\na u,\na u\ra}{|\na u|^2} = \sum_{i=1}^n \lam_i a_i.
\]
Hence, the left hand side of \eqref{eq:claim} becomes
\[
(\be-1) \sum_{i=1}^n \lam_i^2 - \lt(\sum_{i=1}^n \lam_i\rt)^2 + (p-2) \lt((\beta -1) \sum \lam_i^2 a_i - \lt(\sum_{i=1}^n \lam_i\rt) \sum_{i=1}^n \lam_i a_i\rt).
\]
The set $S:= \{x =(x_1, \ldots,x_n)\, :\, x_i \geq 0,\, i =1,\ldots, n,\quad \sum_{i=1}^n x_i =1\}$ is a convex subset of $\R^n$ with extreme points $v_i, i=1,\ldots,n$ such that the $i$th coordinate is $1$ and other coordinates are $0$. The function 
\[
F(x) = (\be-1) \sum_{i=1}^n \lam_i^2 - \lt(\sum_{i=1}^n \lam_i\rt)^2 + (p-2) \lt((\beta -1) \sum \lam_i^2 x_i - \lt(\sum_{i=1}^n \lam_i\rt) \sum_{i=1}^n \lam_i x_i\rt)
\]
is affine on $\R^n$. Hence $\min_{S} F$ is attained at a point $v_i$ for some $i \in \{1, \ldots, n\}$. Let $i_0$ be such an index $i$. Note that $a =(a_1,\ldots,a_n) \in S$, hence  we have 
\begin{align*}
F(a) &\geq (\beta -2)(p-1) \lam_{i_0}^2 + (\beta -1) \sum_{i\not=i_0} \lam_i^2 -p \lam_{i_0} \sum_{i\not=i_0} \lam _i -\lt(\sum_{i\not=i_0} \lam _i\rt)^2\\
&\geq (\beta -2)(p-1) \lam_{i_0}^2 + \frac{\beta -1}{n -1} \lt(\sum_{i\not=i_0} \lam _i\rt)^2-p \lam_{i_0} \sum_{i\not=i_0} \lam _i -\lt(\sum_{i\not=i_0} \lam _i\rt)^2\\
&= (\beta -2)(p-1) \lam_{i_0}^2 + \frac{\beta -n}{n -1} \lt(\sum_{i\not=i_0} \lam _i\rt)^2-p \lam_{i_0} \sum_{i\not=i_0} \lam _i\\
&\geq 2 \lt(\frac{(\beta -2)(\beta -n) (p-1)}{n-1}\rt)^{\frac12} |\lam_{i_0}| \lt|\sum_{i\not=i_0} \lam _i\rt| -p \lam_{i_0} \sum_{i\not=i_0} \lam _i
\end{align*}
here the second and fourth inequalities come from Cauchy--Schwartz inequality. Therefore $F(a) \geq 0$ provided
\[
2 \lt(\frac{(\beta -2)(\beta -n) (p-1)}{n-1}\rt)^{\frac12} \geq p,
\]
for $p\geq 2$. However, this condition is equivalent to our assumption $2 \leq p \leq p_{\beta,n}$. Hence, we have proved
\[
(\be-1) \sum_{i=1}^n \lam_i^2 - \lt(\sum_{i=1}^n \lam_i\rt)^2 + (p-2) \lt((\beta -1) \sum \lam_i^2 a_i - \lt(\sum_{i=1}^n \lam_i\rt) \sum_{i=1}^n \lam_i a_i\rt) = F(a) \geq 0,
\]
for $2 \leq p \leq p_{\beta,n}$. This proves \eqref{eq:claim} when $n \geq 2$.

It follows from \eqref{eq:zero} and \eqref{eq:claim} that
\begin{align}\label{eq:stepfini}
\int_{\R^n} |(D^2\vphi)^{\frac12} \na u|^2 |\na u|^{p-2}& \vphi d\mvb\notag\\
& = \int_{\R^n}\la D^2\vphi \na u, \na u\ra |\na u|^{p-2} \vphi d\mvb\notag\\
& \leq \frac1{\beta -1}\int_{\R^n} \la \na h, \na u\ra |\na u|^{p-2} \vphi d\mvb\notag\\
& = \frac1{\beta -1}\int_{\R^n} \la (D^2\vphi)^{\frac1p} \na u, (D^2\vphi)^{-\frac1p} \na h\ra |\na u|^{p-2}\vphi d\mvb\notag\\
&\leq \frac1{\beta -1}\int_{\R^n} | (D^2\vphi)^{\frac1p} \na u|\, | (D^2\vphi)^{-\frac1p} \na h|\, |\na u|^{p-2} \vphi d\mvb.
\end{align}
Let $A$ be a positive $n\times n$ matrix, and $v$ be a vector in $\R^n$. It is well-known that
\begin{equation}\label{eq:bdtmatrix}
|A^{\frac1p} v|^p \leq |v|^{p-2} |A^{\frac12} v|^2,
\end{equation}
for $p \geq 2$. Moreover, it is obvious that
\begin{equation}\label{eq:bdtmatrix1}
|\na u| \leq \lam_{min}^{-\frac1p} |(D^2\vphi)^{\frac1p} \na u|.
\end{equation}
Inserting the estimates \eqref{eq:bdtmatrix} and \eqref{eq:bdtmatrix1} into \eqref{eq:stepfini} for $A = D^2\vphi$ and $v = \na u$ with notice that $p\geq 2$, we get
\begin{equation}\label{eq:fini}
\int_{\R^n} |(D^2 \vphi)^{\frac1p} \na u|^p \vphi d\mvb \leq \frac1{\beta -1} \int_{\R^n}|(D^2 \vphi)^{\frac1p} \na u|^{p-1} \lam_{min}^{-\frac{p-2}p}| (D^2\vphi)^{-\frac1p} \na h| \vphi d\mvb.
\end{equation}
Applying H\"older inequality to the right hand side of \eqref{eq:fini} and simplifying the obtained inequality, we arrive \eqref{eq:suffice}. The proof of Theorem \ref{asymmetriccov} is completed.
\end{proof}

We conclude this section by showing that the inequality \eqref{eq:CCL} can be derived from our Theorem \ref{asymmetriccov}. Let $\psi$ be a strictly convex function on $\R^n$ such that $\int_{\R^n} e^{-\psi}dx =1$, and $\mu$ be the probability of density $e^{-\psi}$. Perturbing $\psi$ by $\ep |x|^2/2$, we can assume that $\psi$ is uniform convex on $\R^n$. Let $\vphi_\beta = 1 + \frac{\psi}\beta$ and $\mu_{\beta}$ be the probability measure of density $Z_{\beta}^{-1} \psi_\beta^{-\beta}$ for $\beta > n+1$ where $Z_{\beta}$ is normalization constant. We have $D^2\vphi_\beta = \beta^{-1} D^2\psi$. Denote $\lam_{min}$ and $\lam_{min,\beta}$ the smallest eigenvalue of $D^2\psi$ and $D^2\vphi_\beta$ respectively, we have $\lam_{min,\beta} = \beta^{-1} \lam_{min}$. Let $g,h \C_c^\infty(\R^n)$ and $2 \leq p < \infty$. We have $p_{\beta,n} \to \infty$ as $\beta \to \infty$, hence $p_{\beta,n} > p$ for $\beta$ large enough. Applying Theorem \ref{asymmetriccov} to $g, h$ and for $\mu_\beta$ with $\beta$ large enough, we have
\[
|\text{\rm cov}_{\mu_{\beta}}(g,h)| \leq \frac{\beta}{\beta-1}\lt(\int_{\R^n} |(D^2\psi)^{-\frac 1p} \na g|^q \vphi_\beta d\mu_{\beta}\rt)^{\frac1q} \lt(\int_{\R^n} \lam_{min}^{2-p} |(D^2\psi)^{-\frac1p} \na h|^p \vphi_\beta d\mu_{\beta}\rt)^{\frac1p}.
\]
Since $\vphi_\beta \to 1$ and $Z_\beta^{-1} \psi_\beta^{-\beta} \to e^{-\psi}$ as $\beta \to \infty$, then by letting $\beta\to \infty$ in the preceding inequality, we obtain \eqref{eq:CCL} for any function $g, h\in C_c^\infty(\R^n)$. By standard approximation argument, we get \eqref{eq:CCL} for any $2\leq p < \infty$. The case $p =\infty$ is obtained from the case $p < \infty$ by letting $p\to \infty$.
\section*{Acknowledgements}
This work was done when the author was PhD student at the Universit\'e Pierre et Marie Curie (Paris VI) under the supervision of Prof. Dario Cordero--Erausquin. The author would like to thank him for his help and advice.

\end{document}